\documentclass[11pt,leqno]{article}

\usepackage{amsmath,amssymb,amsthm}

\usepackage[margin=3cm]{geometry} %宽版时用，配合双倍行距

\usepackage{titlesec,hyperref}

\usepackage{color}
\usepackage{fancyhdr}
\titleformat{\subsection}{\it}{\thesubsection.\enspace}{1pt}{}

\newtheorem{theo}{Theorem}[section]
\newtheorem{lemm}[theo]{Lemma}
\newtheorem{defi}[theo]{Definition}
\newtheorem{coro}[theo]{Corollary}
\newtheorem{prop}[theo]{Proposition}
\newtheorem{rema}[theo]{Remark}
\numberwithin{equation}{section}

\allowdisplaybreaks %允许公式分页显示

\begin{document}
\title{Blow-up phenomena and local well-posedness for a generalized Camassa-Holm equation with peakon solution
\hspace{-4mm}
}

\author{Xi $\mbox{Tu}^1$\footnote{E-mail: tuxi@mail2.sysu.edu.cn} \quad and\quad
 Zhaoyang $\mbox{Yin}^{1,2}$\footnote{E-mail: mcsyzy@mail.sysu.edu.cn}\\
 $^1\mbox{Department}$ of Mathematics,
Sun Yat-sen University,\\ Guangzhou, 510275, China\\
$^2\mbox{Faculty}$ of Information Technology,\\ Macau University of Science and Technology, Macau, China}

\date{}
\maketitle
\hrule

\begin{abstract}
In this paper we mainly study the Cauchy problem for a generalized Camassa-Holm equation. First, by using the Littlewood-Paley decomposition and transport equations theory, we establish the local well-posedness for the Cauchy problem of the equation in Besov spaces. Then we give a blow-up criterion for the Cauchy problem of the equation. we present a blow-up result and the exact blow-up rate of strong solutions to the equation by making use of the conservation law and the obtained blow-up criterion. Finally, we verify that the system possesses peakon solutions.

\vspace*{5pt}
\noindent {\it 2000 Mathematics Subject Classification}: 35Q53, 35A01, 35B44, 35B65.\\
\vspace*{5pt}
\noindent{\it Keywords}: A generalized Camassa-Holm equation; Local well-posedness; Besov spaces; Blow-up.
\\\end{abstract}

\vspace*{10pt}

%\phantomsection
%\addcontentsline{toc}{section}{\contentsname}
%添加目录到书签
\tableofcontents

\section{Introduction}
In this paper we consider the Cauchy problem for the following generalized Camassa-Holm equation,
\begin{align}\label{E1}
\left\{
\begin{array}{ll}
u_t-u_{txx}=\partial_x(2+\partial_x)[(2-\partial_x)u]^2,~~~~  t>0,\\[1ex]
u(0,x)=u_{0}(x),
\end{array}
\right.
\end{align}
which can be rewritten as
\begin{align}\label{E2}
\left\{
\begin{array}{ll}
m=u-u_{xx},\\[1ex]
m_t=2m^2+(8u_x-4u)m+(4u-2u_x)m_x+2(u+u_x)^2, ~~ t>0,\\[1ex]
m(0,x)=u(0,x)-u_{xx}(0,x)=m_0(x).
\end{array}
\right.
\end{align}
The equation (\ref{E1}) was proposed recently by Novikov in \cite{n1}. It is integrable and belongs to the following class \cite{n1}:
\begin{align}\label{E02}
(1-\partial^2_x)u_t=F(u,u_x,u_{xx},u_{xxx}),
\end{align}
which has attracted much interest, particularly in the possible integrable members of (\ref{E02}).

The most celebrated integrable member of (\ref{E02}) is the well-known Camassa-Holm (CH) equation \cite{Camassa}:
\begin{align}
(1-\partial^2_x)u_t=3uu_x-2u_{x}u_{xx}-uu_{xxx}.
\end{align}
The CH equation can be regarded as a shallow water wave equation \cite{Camassa, Constantin.Lannes}.  It is completely integrable \cite{Camassa,Constantin-P,Constantin.mckean}.
 Integrability is not a straightforward concept, and this provides good background material on this important aspect of
the CH-equation.
It also has a bi-Hamiltonian structure \cite{Constantin-E,Fokas}, and admits exact peaked solitons of the form $ce^{-|x-ct|}$ with $c>0$, which are orbitally stable \cite{Constantin.Strauss}. It is worth mentioning that the peaked solitons present the characteristic for the traveling water waves of greatest height and largest amplitude and arise as solutions to the free-boundary problem for incompressible Euler equations over a flat bed, cf. \cite{Camassa.Hyman,Constantin2,Constantin.Escher4,Constantin.Escher5,Toland}.

The local well-posedness for the Cauchy problem of the CH equation in Sobolev spaces and Besov spaces was discussed in \cite{Constantin.Escher,Constantin.Escher2,d1,Guillermo}. It was shown that there exist global strong solutions to the CH equation \cite{Constantin,Constantin.Escher,Constantin.Escher2} and finite time blow-up strong solutions to the CH equation \cite{Constantin,Constantin.Escher,Constantin.Escher2,Constantin.Escher3}. The existence and uniqueness of global weak solutions to the CH equation were proved in \cite{Constantin.Molinet, Xin.Z.P}. The global conservative and dissipative solutions of CH equation were investigated in \cite{Bressan.Constantin,Bressan.Constantin2}.

The second celebrated integrable member of (\ref{E02}) is the famous Degasperis-Procesi (DP) equation \cite{D-P}:
\begin{align}
(1-\partial^2_x)u_t=4uu_x-3u_{x}u_{xx}-uu_{xxx}.
\end{align}
The DP
equation can be regarded as a model for nonlinear shallow water
dynamics and its asymptotic accuracy is the same as for the
CH shallow water equation \cite{D-G-H}. The DP equation is integrable and has a bi-Hamiltonian structure \cite{D-H-H}. An inverse scattering approach for the
DP equation was presented in \cite{Constantin.lvanov.lenells,Lu-S}. Its
traveling wave solutions was investigated in \cite{Le,V-P}. \par The
local well-posedness of the Cauchy problem of the DP equation in Sobolev spaces and Besov spaces was established in
\cite{G-L,H-H,y1}. Similar to the CH equation, the
DP equation has also global strong solutions
\cite{L-Y1,y2,y4} and finite time blow-up solutions
\cite{E-L-Y1, E-L-Y,L-Y1,L-Y2,y1,y2,y3,y4}. On the other hand, it has global weak
solutions \cite{C-K,E-L-Y1,y3,y4}.
\par
Although the DP equation is similar to the
CH equation in several aspects, these two equations are
truly different. One of the novel features of the DP
different from the CH equation is that it has not only
peakon solutions \cite{D-H-H} and periodic peakon solutions
\cite{y3}, but
also shock peakons \cite{Lu} and the periodic shock waves \cite{E-L-Y}.

The third celebrated integrable member of (\ref{E02}) is the known Novikov equation \cite{n1}:
\begin{align}
(1-\partial^2_x)u_t=3uu_{x}u_{xx}+u^2u_{xxx}-4u^2u_x.
\end{align}
The most difference between the Novikov equation and the CH and DP equations is that the former one has cubic nonlinearity and the latter ones have quadratic nonlinearity.

It was showed that the Novikov equation is integrable, possesses a bi-Hamiltonian structure, and admits exact peakon solutions $u(t,x)=\pm\sqrt{c}e^{|x-ct|}$ with $c>0$ \cite{Hone}.\\
$~~~~~~$ The local well-posedness for the Novikov equation in Sobolev spaces and Besov spaces was studied in \cite{Wu.Yin2,Wu.Yin3,Wei.Yan,Wei.Yan2}. The global existence of strong solutions was established in \cite{Wu.Yin2} under some sign conditions and the blow-up phenomena of the strong solutions were shown in \cite{Wei.Yan2}. The global weak solutions for the Novikov equation were studied in \cite{Laishaoyong,Wu.Yin}.\\

To our best knowledge, the Cauchy problem of (\ref{E1}) has not been studied yet. In this paper we first investigate the local well-posedness of (\ref{E2}) with initial data in Besov spaces $B^s_{p,r}$ with $s>\max\{\frac{1}{2},\frac{1}{p}\}$. The main idea is based on the Littlewood-Paley theory and transport equations theory. Then, we prove a blow-up criterion with the help of the Kato-Ponce commutator estimate. By virtue of conservation laws, we obtain a blow-up result. Finally, we conclude the exact blow-up rate of strong solutions to (\ref{E1}). Finally, we verify new peakon solutions of (\ref{E1}) in distributional sense.

The paper is organized as follows. In Section 2 we introduce some preliminaries which will be used in sequel. In Section 3 we prove the local well-posedness of (\ref{E1}) by using Littlewood-Paley and transport equations theory. In section 4, we are committed to the study of blow-up phenomena of (\ref{E1}). Taking advantage of a conservation law and a prior estimates, we derive a blow-up criterion, a blow-up result and the exact blow-up rate of strong solutions to (\ref{E1}). The last section is devoted to the study of the equation (\ref{E1}) possessing a class of peakon solutions.

\vspace*{2em}
\section{Preliminaries}
In this section, we first recall the Littlewood-Paley decomposition and Besov spaces (for more details to see \cite{B.C.D}).
Let $\mathcal{C}$ be the annulus $\{\xi\in\mathbb{R}^{d}\big|\frac{3}{4}\leq|\xi|\leq\frac{8}{3}\}.$ There exist radial functions $\chi$ and $\varphi$, valued in the interval $[0,1]$, belonging respectively to $\mathcal{D}(B(0,\frac{4}{3}))$ and $\mathcal{D}(\mathcal{C})$, and such that $$\forall\xi\in\mathbb{R}^{d},~\chi(\xi)+\sum_{j\geq0}\varphi(2^{-j}\xi)=1,$$

$$|j-j'|\geq2\Rightarrow Supp ~\varphi(2^{-j}\xi)\cap Supp ~\varphi(2^{-j'}\xi)=\emptyset,$$

$$j\geq1\Rightarrow Supp ~\chi(\xi)\cap Supp ~\varphi(2^{-j'}\xi)=\emptyset.$$
  Define the set $\widetilde{\mathcal{C}}=B(0,\frac{2}{3})+\mathcal{C}$. Then we have
$$|j-j'|\geq5\Rightarrow 2^{j'}\widetilde{\mathcal{C}}\cap 2^{j}\mathcal{C}=\emptyset.$$
Further, we have
 $$\forall\xi\in\mathbb{R}^{d},~\frac{1}{2}\leq\chi^{2}(\xi)+\sum_{j\geq0}\varphi^{2}(2^{-j}\xi)\leq1.$$

Denote $\mathcal{F}$ by the Fourier transform and $\mathcal{F}^{-1}$ by its inverse. From now on, we write $h=\mathcal{F}^{-1}\varphi$ and $\widetilde{h}=\mathcal{F}^{-1}\chi$.
The nonhomogeneous dyadic blocks $\Delta_{j}$ are defined by
$$\Delta_{j}u=0~~~ if~~~ j\leq-2,~~~\Delta_{-1}u=\chi(D)u=\int_{\mathbb{R}^{d}}\widetilde{h}(y)u(x-y)dy,$$
$$ and,~~~\Delta_{j}u=\varphi(2^{-j}D)u=2^{jd}\int_{\mathbb{R}^{d}}h(2^{j}y)u(x-y)dy ~~~if~~ j\geq0,$$
$$S_{j}u=\sum_{j'\leq j-1}\Delta_{j'}u.$$
The nonhomogeneous Besov spaces are denoted by $B^{s}_{p,r}(\mathbb{R}^d)$
$$B^{s}_{p,r}=\big\{u\in S'\big{|}\|u\|_{B^{s}_{p,r}(\mathbb{R}^d)}=(\sum_{j\geq-1}2^{rjs}\|\Delta_{j}u\|^{r}_{L^{p}(\mathbb{R}^d)})^{\frac{1}{r}}<\infty\big\}.$$
Next we introduce some useful lemmas and propositions about Besov spaces which will be used in the sequel.

\begin{prop}\label{2}
\cite{B.C.D} Let $1\leq p_{1} \leq p_{2} \leq \infty$ and $1\leq r_{1} \leq r_{2} \leq \infty$, and let $s$ be a real number. Then we have
$$B^{s}_{p_{1},r_{1}}(\mathbb{R}^d)\hookrightarrow B^{s-d(\frac{1}{p_{1}}-\frac{1}{p_{2}})}_{p_{2},r_{2}}(\mathbb{R}^d).$$
If $s>\frac{d}{p}~or ~s=\frac{d}{p},~r=1$, we then have $$B^{s}_{p,r}(\mathbb{R}^d)\hookrightarrow L^{\infty}(\mathbb{R}^d).$$
\end{prop}

\begin{lemm}\label{3}\cite{B.C.D}
A constant C exists which satisfies the following properties.
If $s_1$ and $s_2$ are real numbers such that $s_1<s_2$ and $\theta \in (0, 1)$, then we
have, for any $(p,r)\in [1,\infty]^2$ and $u\in\mathcal{S}_{h}',$
\begin{align}
&\|u\|_{B_{p,r}^{\theta s_1+(1-\theta)s_2}}\leq\|u\|^{\theta}_{B_{p,r}^{s_1}}
\|u\|^{(1-\theta)}_{B_{p,r}^{s_2}}~~~~~~and\\&
\|u\|_{B_{p,1}^{\theta s_1+(1-\theta)s_2}}\leq\frac{C}{s_2-s_1}(\frac{1}{\theta}+\frac{1}{1-\theta})\|u\|^{\theta}_{B_{p,\infty}^{s_1}}
\|u\|^{(1-\theta)}_{B_{p,\infty}^{s_2}}.
\end{align}
\end{lemm}

\begin{coro}\label{est3}
\cite{B.C.D} For any positive real number $s$ and any $(p, r)$ in $[1,\infty]^{2}$, the
space $L^{\infty}(\mathbb{R}^d)\cap B^{s}_{p,r}(\mathbb{R}^d)$ is an algebra, and a constant $C$ exists such that
$$\|uv\|_{B^{s}_{p,r}(\mathbb{R}^d)}\leq C(\|u\|_{L^{\infty}(\mathbb{R}^d)}\|v\|_{B^{s}_{p,r}(\mathbb{R}^d)}+\|u\|_{B^{s}_{p,r}(\mathbb{R}^d)}\|v\|_{L^{\infty}(\mathbb{R}^d)}).$$
If $s>\frac{d}{p}$ or $s=\frac{d}{p},~r=1$, then we have
$$\|uv\|_{B^{s}_{p,r}(\mathbb{R}^d)}\leq C\|u\|_{B^{s}_{p,r}(\mathbb{R}^d)}\|v\|_{B^{s}_{p,r}(\mathbb{R}^d)}.$$
\end{coro}

\begin{lemm}\label{Morse}
(Morse-type estimate, \cite{B.C.D,d1}) Let $s>\max\{\frac{d}{p},\frac{d}{2}\}$ and $(p, r)$ in $[1,\infty]^{2}$. For any $a\in B^{s-1}_{p,r}(\mathbb{R}^d)$ and $b\in B^{s}_{p,r}(\mathbb{R}^d)$, there exists a constant $C$ such that
$$\|ab\|_{B^{s-1}_{p,r}(\mathbb{R}^d)}\leq C\|a\|_{B^{s-1}_{p,r}(\mathbb{R}^d)}\|b\|_{B^{s}_{p,r}(\mathbb{R}^d)}.$$
\end{lemm}

\begin{rema}
\cite{B.C.D} Let $s\in\mathbb{R},1\leq p,r\leq\infty$. Then the following properties hold true: \\
$(i)~B^s_{p,r}(\mathbb{R}^d)$ is a Banach space and continuously embedding into $\mathcal{S}'(\mathbb{R}^d)$, where $\mathcal{S}'(\mathbb{R}^d)$ is the dual space of the Schwartz space $\mathcal{S}(\mathbb{R}^d)$. \\
$(ii)$ If $p,r<\infty$, then $\mathcal{S}(\mathbb{R}^d)$ is dense in $B^s_{p,r}(\mathbb{R}^d)$.\\
$(iii)$ If $u_n$ is a bounded sequence of $B^s_{p,r}(\mathbb{R}^d)$, then an element $u\in B^s_{p,r}(\mathbb{R}^d)$ and a subsequence $u_{n_k}$ exist such that
$$ \lim_{k\rightarrow\infty}u_{n_k}=u~~in~~\mathcal{S}'(\mathbb{R}^d)~~and~~\|u\|_{B^s_{p,r}(\mathbb{R}^d)}\leq C\liminf_{k\rightarrow\infty}\|u_{n_k}\|_{B^s_{p,r}(\mathbb{R}^d)}.$$
$(iv)~~B^s_{2,2}(\mathbb{R}^d)=H^s(\mathbb{R}^d)$.
\end{rema}

Now we introduce a priori estimates for the following transport equation.
 \begin{align}\label{20}
\left\{
\begin{array}{ll}
f_{t}+v\nabla f=g,\\[1ex]
f|_{t=0}=f_{0}.\\[1ex]
\end{array}
\right.
\end{align}
\begin{lemm}\label{est1}
(A priori estimates in Besov spaces, \cite{B.C.D,d1}) Let $1\leq p \leq p_1\leq \infty$, $1\leq r\leq \infty$, $s\geq -d\min(\frac{1}{p_1},\frac{1}{p'})$. For the solution $f\in L^{\infty}([0,T];B^s_{p,r}(\mathbb{R}^d))$ of (\ref{20}) with velocity $\nabla v\in L^1([0,T];B^s_{p,r}(\mathbb{R}^d)\cap L^{\infty}(\mathbb{R}^d))$,  initial data $f_0\in B^s_{p,r}(\mathbb{R}^d)$ and $g\in L^1([0,T];B^s_{p,r}(\mathbb{R}^d))$, we have the following statements.
If $s\neq 1+{1\over p}$ or $r=1$,
\begin{align}\label{9}
\|f(t)\|_{B^{s}_{p,r}(\mathbb{R}^d)}\leq \|f_0\|_{B^s_{p,r}(\mathbb{R}^d)}+\int^t_0\bigg(\|g(t')\|_{B^s_{p,r}(\mathbb{R}^d)}
+CV'_{p_1}(t')\|f(t')\|_{B^s_{p,r}(\mathbb{R}^d)}\bigg)dt',
\end{align}
\begin{align}\label{10}
\|f\|_{B^{s}_{p,r}(\mathbb{R}^d)}\leq \bigg(\|f_0\|_{B^s_{p,r}(\mathbb{R}^d)}+\int^t_0\exp(-CV_{p_1}(t'))
\|g(t')\|_{B^s_{p,r}(\mathbb{R}^d)}dt'\bigg)\exp(CV_{p_1}(t)),
\end{align}
where  $V_{p_1}(t)=\displaystyle\int^t_0\|\nabla v\|_{B^{\frac{d}{p_1}}_{p_1,\infty}(\mathbb{R}^d)\cap L^{\infty}(\mathbb{R}^d)}dt'$ if $s<1+\frac{d}{p_1}$, $V_{p_1}(t)=\displaystyle\int^t_0\|\nabla v\|_{B^{s-1}_{p_1,r}(\mathbb{R}^d)}dt'$ if $s>1+\frac{d}{p_1}$ or $s=1+\frac{d}{p_1}, r=1$, and $C$ is a constant depending only on $s,~p,~p_1$ and $r$.
\end{lemm}

\begin{lemm}\label{est0}\cite{Luo-Yin}
Let $1\leq p\leq \infty$, $1\leq r\leq \infty$, $\sigma> \max(\frac{1}{2},\frac{1}{p})$. For the solution $f\in L^{\infty}(0,T;B^\sigma_{p,r}(\mathbb{R}))$ of (\ref{20}) with the velocity $ v\in L^1(0,T;B^{\sigma+1}_{p,r}(\mathbb{R}))$,  the initial data $f_0\in B^\sigma_{p,r}(\mathbb{R})$ and $g\in L^1(0,T;B^\sigma_{p,r}(\mathbb{R}^d))$, we have
\begin{align}\label{010}
\|f\|_{B^{\sigma-1}_{p,r}(\mathbb{R})}\leq \bigg(\|f_0\|_{B^{\sigma-1}_{p,r}(\mathbb{R})}+\int^t_0\exp(-CV(t'))\|g(t')\|_{B^{\sigma-1}_{p,r}(\mathbb{R})}dt'\bigg)\exp(CV(t)),
\end{align}
where  $V(t)=\displaystyle\int^t_0\|v\|_{B^{\sigma+1}_{p,r}(\mathbb{R})}$ and $C$ is a constant depending only on $\sigma,~p$ and $r$.
\end{lemm}

\begin{lemm}\label{est3}\cite{Luo-Yin}
 For the solution $f\in L^{\infty}(0,T;B^{1+\frac{1}{p}}_{p,r}(\mathbb{R}))$ of (\ref{20}) with the velocity $ v\in L^1(0,T;B^{2+\frac{1}{p}}_{p,r}(\mathbb{R}))$,  the initial data $f_0\in B^{1+\frac{1}{p}}_{p,r}(\mathbb{R})$ and $g\in L^1(0,T;B^{1+\frac{1}{p}}_{p,r}(\mathbb{R}^d))$, we have
\begin{align}
\|f\|_{B^{1+\frac{1}{p}}_{p,r}(\mathbb{R})}\leq \bigg(\|f_0\|_{B^{1+\frac{1}{p}}_{p,r}(\mathbb{R})}+\int^t_0\exp(-CV(t'))\|g(t')\|_{B^{1+\frac{1}{p}}_{p,r}(\mathbb{R})}dt'\bigg)\exp(CV(t)),
\end{align}
where  $V(t)=\displaystyle\int^t_0\|v\|_{B^{2+\frac{1}{p}}_{p,r}(\mathbb{R})}$ and $C$ is a constant depending only on $p$ and $r$.
\end{lemm}

\begin{lemm}\label{est2}\cite{B.C.D}
Let $s$ be as in the statement of Lemma \ref{est1}. Let $f_0 \in B^s_{p,r}(\mathbb{R}^d)$, $g \in L^1([0, T];B^s_{p,r}(\mathbb{R}^d))$, and $v$ be a time-dependent vector field such that $v\in L^\rho([0, T];B^{-M}_{\infty,\infty}(\mathbb{R}^d))$ for some $\rho > 1$ and $M >0$, and
$$\nabla v\in L^1([0, T];B_{p_1,\infty}^{\frac{d}{p}}(\mathbb{R}^d) ), ~~if ~s<1+\frac{d}{p_1}, $$
$$\nabla v\in L^1([0, T];B_{p_1,\infty}^{s-1}(\mathbb{R}^d)), ~~if~ s>1+\frac{d}{p_1}
~~or~~s=1+\frac{d}{p_1}~ and~ r=1 .$$
Then, (\ref{20}) has a unique solution $f$ in
\\-the space $\mathcal{C}([0,T];B_{p,r}^s(\mathbb{R}^d)),~~if~ r<\infty,$
\\-the space $(\bigcap_{s'<s}\mathcal{C}([0,T];B_{p,\infty}^{s'}
(\mathbb{R}^d)))\bigcap\mathcal{C}_{w}([0,T];B_{p,\infty}^{s}(\mathbb{R}^d))),~~if~ r=\infty.$
\\Moreover, the inequalities of Lemma \ref{est1} hold true.
\end{lemm}

\begin{lemm}(Kato-Ponce commutator estimates, \cite{K.P})\label{com.est}.
If $s>0$, $f\in H^s(\mathbb{R})\cap W^{1,\infty}(\mathbb{R}),~g\in H^{s-1}(\mathbb{R})\cap L^\infty(\mathbb{R})$ and denote that $\Lambda^s=(1-\Delta)^{\frac{s}{2}}$, then
$$\|\Lambda^s(fg)-f\Lambda^s g\|_{L^2(\mathbb{R})}\leq C(\|\Lambda^sf\|_{L^2(\mathbb{R})}\|g\|_{L^\infty(\mathbb{R})}+\|f_x\|_{L^\infty(\mathbb{R})}\|\Lambda^{s-1}g\|_{L^2(\mathbb{R})}).$$
\end{lemm}

{\bf
Notations.}
Since all space of functions in the following sections are over $\mathbb{R},$ for simplicity, we drop $\mathbb{R}$ in our notations of function spaces if there is no ambiguity.

\section{Local well-posedness}
In this section, we establish local well-posedness of (\ref{E2}) in the Besov spaces. Our main result can be stated as follows. To introduce the main result, we define
 \begin{align}
   E^s_{p,r}(T)\triangleq
\left\{
\begin{array}{ll}
C([0,T);B^s_{p,r})\cap C^1([0,T);B^{s-1}_{p,r}),~~~~if~r<\infty, \\[1ex]
C_w([0,T);B^s_{p,\infty})\cap C^{0,1}([0,T);B^{s-1}_{p,\infty}),~~~~if~r=\infty.\\[1ex]
\end{array}
\right.
\end{align}
\begin{theo}\label{Thm1}
Let $1\leq p,~r\leq \infty,~s>\max\{\frac{1}{p},\frac{1}{2}\},$ and  $~m_0\in B^s_{p,r}.$ Then there exists some $T>0$, such that (\ref{E2}) has a unique solution $u$ in
$ E^s_{p,r}(T).$ Moveover the solution depends continuously on the initial data $m_{0}.$
\end{theo}
\begin{proof}
In order to prove Theorem \ref{Thm1}, we proceed as the following six steps.\\
 {\bf Step 1:} First, we construct approximate solutions which are smooth solutions of some linear equations. Starting for $m_0(t,x)\triangleq m(0,x)=m_0$, we define by induction sequences $(m_{n})_{n\in\mathbb{N}}$  by solving the following linear transport equations:
 \begin{align}\label{E00}
\left\{
\begin{array}{ll}
 \partial_{t}m_{n+1}-(4u_{n}-2\partial_{x}u_{n})\partial_{x}m_{n+1}&=2m_{n}^2
 +(8\partial_xu_n-4u_n)m_{n}+2(u_{n}+\partial_xu_{n})^2\\&=F(m_{n},u_{n}),\\[1ex]
 m_{n+1}(t,x)|_{t=0}=S_{n+1}m_{0}.\\[1ex]
\end{array}
\right.
\end{align}

We assume that $m_n\in L^{\infty}(0,T;B^{s}_{p,r}),~~s>\max\{\frac{1}{p},\frac{1}{2}\}$.

Since $s>\max\{\frac{1}{p},\frac{1}{2}\}$, it follows that $B^{s}_{p,r}$ is an algebra, which leads to
$F(m_{n},u_{n})\in L^{\infty}(0,T;B^{s}_{p,r})$.
Hence, by Lemma \ref{est2} and the high regularity of $u$, (\ref{E00}) has a global solution $m_{n+1}$ which belongs to $E^s_{p,r}$ for all positive $T$.

{\bf Step 2:} Next, we are going to find some positive $T$ such that for this fixed $T$ the approximate solutions are uniformly bounded on $[0,T]$.
We define that
$U_{n}(t)\triangleq\int^{t}_{0}\|m_n(t')\|_{B^{s}_{p,r}}dt'$. By Lemma \ref{est1} and Lemma \ref{est3}, we infer that
\begin{align}\label{11}
\|m_{n+1}\|_{B^{s}_{p,r}}
\nonumber\leq& e^{C\int_{0}^{t}\|\partial_{x}(4u_n-2\partial_{x}u_n)\|_{B_{p,r}^{s}}}
\bigg(\|S_{n+1}m_0\|_{B^s_{p,r}}
\\\nonumber&+\int^t_0 e^{-C\int_{0}^{t'}\|\partial_{x}(4u_n-2\partial_{x}u_n)\|_{B_{p,r}^{s}}}
\|F(m_{n},u_{n})\|_{B^s_{p,r}}dt'\bigg)
\\ \leq& e^{CU_{n}(t)}\bigg(\|S_{n+1}m_0\|_{B^s_{p,r}}
+\int^t_0e^{-CU_{n}(t')}\|F(m_{n},u_{n})\|_{B^s_{p,r}}dt'\bigg).
\end{align}
Since $s>\frac{1}{p}$, $B^{s}_{p,r}$ is an algebra and $B^{s}_{p,r} \hookrightarrow L^{\infty}$, we deduce that
\begin{align}\label{12}
&\nonumber\|2m_{n}^2+(8\partial_xu_n-4u_n)m_{n}+2(u_{n}+\partial_xu_{n})^2\|_{B^{s}_{p,r}}
\\\nonumber \leq&
\|2m_{n}^2\|_{B^{s}_{p,r}}+\|(8\partial_xu_n-4u_n)m_{n}\|_{B^{s}_{p,r}}
+\|2(u_{n}+\partial_xu_{n})^2\|_{B^{s}_{p,r}}
\\\nonumber \leq&
C\|m_{n}\|_{B^{s}_{p,r}}\|m_{n}\|_{L^{\infty}}+C\|m_{n}\|_{B^{s}_{p,r}}\|8\partial_xu_n-4u_n\|_{L^{\infty}}
\\\nonumber&+C\|8\partial_xu_n-4u_n\|_{B^{s}_{p,r}}\|m_{n}\|_{L^{\infty}}
+C\|u_{n}+\partial_xu_{n}\|_{B^{s}_{p,r}}\|u_{n}+\partial_xu_{n}\|_{L^{\infty}}
\\ \leq&C\|m_{n}\|^2_{B^{s}_{p,r}}.
\end{align}
Plugging (\ref{12}) into (\ref{11}), we obtain
\begin{align}\label{14}
\nonumber\|m_{n+1}\|_{B^{s}_{p,r}}\leq& e^{CU_{n}(t)}\bigg(\|S_{n+1}m_0\|_{B^s_{p,r}}
+C\int^t_0e^{-CU_{n}(t')}\|m_n\|^2_{B^{s}_{p,r}}dt'\bigg)
\\\leq& e^{CU_{n}(t)}\bigg(C\|m_0\|_{B^s_{p,r}}
+C\int^t_0e^{-CU_{n}(t')}\|m_n\|^2_{B^{s}_{p,r}}dt'\bigg),
\end{align}
where we take $C\geq1.$
\\We fix a $T>0$ such that $2C^2T\|m_0\|_{B^{s}_{p,r}}<1.$ Suppose that
\begin{align}\label{013}
\|m_{n}(t)\|_{B^{s}_{p,r}}\leq \frac{C\|m_0\|_{B^s_{p,r}}}{1-2C^2\|m_0\|_{B^s_{p,r}}t}
\leq \frac{C\|m_0\|_{B^s_{p,r}}}{1-2C^2\|m_0\|_{B^s_{p,r}}T}\triangleq \mathbf{M},~~~~\forall t\in[0,T].
\end{align}

Since $U_{n}(t)=\int^{t}_{0}\|m_n(\tau)\|_{B^{s}_{p,r}}d\tau$, it follows that
\begin{align}\label{13}
e^{CU_{n}(t)-CU_{n}(t')}\nonumber\leq& \exp\bigg{\{}\int^{t}_{t'}\frac{C^2\|m_0\|_{B^s_{p,r}}}{1-2C^2\|m_0\|_{B^s_{p,r}}t} d\tau \bigg{\}}
\\\nonumber\leq&\exp\bigg{\{}-\frac{1}{2}\int^{t}_{t'}
\frac{d(1-2C^2\tau\|m_0\|_{B^s_{p,r}})}{1-2C^2\tau\|m_0\|_{B^s_{p,r}}} \bigg{\}}
\\=& (\frac{1-2C^2t'\|m_0\|_{B^s_{p,r}}}{1-2C^2t\|m_0\|_{B^s_{p,r}}})^{\frac{1}{2}}.
\end{align}
Set $U_{n}(t')=0$ when $t'=0$. We obtain
\begin{align}\label{14}
e^{CU_{n}(t)}\nonumber=& \exp\bigg{\{}C^2\int^{t}_{0}\frac{\|m_0\|_{B^s_{p,r}}}
{1-2C^2\|m_0\|_{B^s_{p,r}}\tau} d\tau \bigg{\}}
\\\nonumber\leq&\exp\bigg{\{}-\frac{1}{2}\int^{t}_{0}
\frac{d(1-2C^2\tau\|m_0\|_{B^s_{p,r}})}{1-2C^2\tau\|m_0\|_{B^s_{p,r}}} \bigg{\}}
\\=& (\frac{1}{1-2C^2t\|m_0\|_{B^s_{p,r}}})^{\frac{1}{2}}.
\end{align}
By using (\ref{013}), (\ref{13}) and (\ref{14}), we have
\begin{align}\label{8}
\|m_{n+1}(t)\|_{B^{s}_{p,r}}\nonumber &\leq Ce^{CU_{n}(t)}\|m_0\|_{_{B^s_{p,r}}}
+C\int^t_0e^{CU_{n}(t)-CU_{n}(t')}\|m_n(t')\|_{B^s_{p,r}}^{2}dt'
\\\nonumber& \leq (\frac{1}{1-2C^2t\|m_0\|_{B^s_{p,r}}^{2}})^{\frac{1}{2}}
\bigg{\{}C\|m_0\|_{B^s_{p,r}}+\int^t_0(\frac{C^3\|m_0\|^{2}_{B^s_{p,r}}}
{(1-2C^2t'\|m_0\|_{B^s_{p,r}})^{1+\frac{1}{2}}})dt'\bigg{\}}
\\\nonumber& \leq  (\frac{1}{1-2bC^2t\|m_0\|_{B^s_{p,r}}})^{\frac{1}{2}}
\bigg{\{}C\|m_0\|_{B^s_{p,r}}-\frac{C\|m_0\|_{B^s_{p,r}}}{2}
\int^t_0\frac{d(1-2C^2t'\|m_0\|_{B^s_{p,r}})}
{(1-2C^2t'\|m_0\|_{B^s_{p,r}})^{1+\frac{1}{2}}}\bigg{\}}
\\\nonumber&  \leq(\frac{1}{1-2C^2t\|m_0\|_{B^s_{p,r}}})^{\frac{1}{2}}
\bigg{\{}C\|m_0\|_{B^s_{p,r}}+C\|m_0\|_{B^s_{p,r}}
(\frac{1}{1-2C^2t'\|m_0\|_{B^s_{p,r}}})^{\frac{1}{2}}|^t_0\bigg{\}}
\\\nonumber&= \frac{C\|m_0\|_{B^s_{p,r}}}{1-2bC^2t\|m_0\|_{B^s_{p,r}}}
\\&\leq \frac{C\|m_0\|_{B^s_{p,r}}}{1-2C^2T\|m_0\|_{B^s_{p,r}}}=\mathbf{M}.
\end{align}
Thus, $(m_{n})_{n \in \mathbb{N}}$ is uniformly bounded in $L^{\infty}(0,T; B^s_{p,r})$.

{\bf Step 3:} From now on, we are going to prove that $u_n$ is a Cauchy sequence
 in $ L^{\infty}(0,T;B^{s-1}_{p,r})$. For this purpose, we deduce from (\ref{E00}) that

\begin{align}\label{E01}
\left\{
\begin{array}{ll}
&\partial_{t}(m_{n+l+1}-m_{n+1})-(4u_{n+l}-2\partial_{x}u_{n+l})\partial_{x}(m_{n+l+1}-m_{n+1})
 \\&=(u_{n+l}-u_{n})(\partial_{x}R_{n,l}^{1}+R_{n,l}^{2})
 +\partial_{x}(u_{n+l}-u_{n})
 (\partial_{x}R_{n,l}^{3}+R_{n,l}^{4})
 \\&~~+(m_{n+l}-m_{n})R_{n,l}^{5},\\[1ex]
 &m_{n+l+1}(t,x)-m_{n+1}(t,x)|_{t=0}=(S_{n+l+1}-S_{n+1})m_{0},\\[1ex]
\end{array}
\right.
\end{align}
where
\begin{align}
\nonumber&R_{n,l}^{1}=4m_{n+1}+2u_{n+l}+2u_{n},
\\\nonumber&R_{n,l}^{2}=-4m_{n}+2u_{n+l}+2u_{n},
\\\nonumber&R_{n,l}^{3}=-2m_{n+1}+2u_{n+l}+2u_{n},
\\\nonumber&R_{n,l}^{4}=8m_{n}+2u_{n+l}+2u_{n},
\\\nonumber&R_{n,l}^{5}=2m_{n+l}+2m_{n}+8\partial_{x}u_{n+l}-4u_{n+l}.
\end{align}

By  Lemma \ref{est0} and Lemma \ref{est3} and using the fact that $m_n$ is bounded in $L^{\infty}(0,T;B^{s}_{p,r})$, we infer that
\begin{align}\label{015}
\nonumber\|m_{n+m+1}(t)-m_{n+1}(t)\|_{B^{s-1}_{p,r}}\leq& C\bigg(\|(S_{n+m+1}-S_{n+1})m_0\|_{B^{s-1}_{p,r}}
\\\nonumber&+\int^t_0
\|(u_{n+l}-u_{n})\partial_{x}R_{n,l}^{1}\|_{B^{s-1}_{p,r}}
+\|(u_{n+l}-u_{n})R_{n,l}^{2}\|_{B^{s-1}_{p,r}}
\\\nonumber&+\|\partial_{x}(u_{n+l}-u_{n})\partial_{x}R_{n,l}^{3}\|_{B^{s-1}_{p,r}}
+\|\partial_{x}(u_{n+l}-u_{n})R_{n,l}^{4}\|_{B^{s-1}_{p,r}}
\\&+\|(m_{n+l}-m_{n})R_{n,l}^{5}\|_{B^{s-1}_{p,r}}dt'\bigg).
\end{align}

Applying Lemma \ref{Morse} with $d=1$, we have
\begin{align}\label{016}
\|(u_{n+l}-u_{n})\partial_{x}R_{n,l}^{1}\|_{B^{s-1}_{p,r}}\nonumber&\leq
\|u_{n+l}-u_{n}\|_{B^{s}_{p,r}}\|\partial_{x}(4m_{n+1}+2u_{n+l}+2u_{n})\|_{B^{s-1}_{p,r}}\\\nonumber&\leq
\|m_{n+l}-m_{n}\|_{B^{s-1}_{p,r}}\|4m_{n+1}-2u_{n+l}-2u_{n}\|_{B^{s}_{p,r}}\\\nonumber&\leq
C\|m_{n+l}-m_{n}\|_{B^{s-1}_{p,r}}(\|m_{n+1}\|_{B^{s}_{p,r}}
+\|u_{n+l}\|_{B^{s}_{p,r}}+\|u_{n}\|_{B^{s}_{p,r}})
\\&\leq C\mathbf{M}\|m_{n+l}-m_{n}\|_{B^{s-1}_{p,r}},
\\
\|(u_{n+l}-u_{n})R_{n,l}^{2}\|_{B^{s-1}_{p,r}}&\nonumber\leq
\|u_{n+l}-u_{n}\|_{B^{s-1}_{p,r}}\|-4m_{n}+2u_{n+l}+2u_{n}\|_{B^{s}_{p,r}}
\\\nonumber&\leq C\|m_{n+l}-m_{n}\|_{B^{s-1}_{p,r}}(\|m_{n}\|_{B^{s}_{p,r}}
+\|u_{n+l}\|_{B^{s}_{p,r}}+\|u_{n}\|_{B^{s}_{p,r}})
\\&\leq C\mathbf{M}\|m_{n+l}-m_{n}\|_{B^{s-1}_{p,r}},
\\
\|\partial_{x}(u_{n+l}-u_{n})\partial_{x}R_{n,l}^{3}\|_{B^{s-1}_{p,r}}\nonumber&\leq
\|\partial_{x}(u_{n+l}-u_{n})\|_{B^{s}_{p,r}}\|\partial_{x}(-2m_{n+1}+2u_{n+l}+2u_{n})\|_{B^{s-1}_{p,r}}
\\ \nonumber&\leq
 C\|m_{n+l}-m_{n}\|_{B^{s-1}_{p,r}}(\|m_{n+1}\|_{B^{s}_{p,r}}
+\|u_{n+l}\|_{B^{s}_{p,r}}+\|u_{n}\|_{B^{s}_{p,r}})
\\&\leq C\mathbf{M}\|m_{n+l}-m_{n}\|_{B^{s-1}_{p,r}},
\\
\|\partial_{x}(u_{n+l}-u_{n})R_{n,l}^{4}\|_{B^{s-1}_{p,r}}\nonumber&\leq
\|\partial_{x}(u_{n+l}-u_{n})\|_{B^{s-1}_{p,r}}\|8m_{n}+2u_{n+l}+2u_{n}\|_{B^{s}_{p,r}}
\\ \nonumber&\leq
 C\|m_{n+l}-m_{n}\|_{B^{s-1}_{p,r}}(\|m_{n}\|_{B^{s}_{p,r}}
+\|u_{n+l}\|_{B^{s}_{p,r}}+\|u_{n}\|_{B^{s}_{p,r}})
\\&\leq C\mathbf{M}\|m_{n+l}-m_{n}\|_{B^{s-1}_{p,r}},
\end{align}
\begin{align}\label{017}
\|(m_{n+l}-m_{n})R_{n,l}^{5}\|_{B^{s-1}_{p,r}}\nonumber\leq&
\|m_{n+l}-m_{n})\|_{B^{s-1}_{p,r}}\|2m_{n+l}+2m_{n}+8\partial_{x}u_{n+l}-4u_{n+l}\|_{B^{s}_{p,r}}
\\\nonumber \leq
& C\|m_{n+l}-m_{n}\|_{B^{s-1}_{p,r}}(\|m_{n+l}\|_{B^{s}_{p,r}}
+\|\partial_{x}u_{n+l}\|_{B^{s}_{p,r}}+\|u_{n+l}\|_{B^{s}_{p,r}})
\\ \nonumber\leq
& C\|m_{n+l}-m_{n}\|_{B^{s-1}_{p,r}}\|m_{n+l}\|_{B^{s}_{p,r}}
\\ \leq& C\mathbf{M}\|m_{n+l}-m_{n}\|_{B^{s-1}_{p,r}}.
\end{align}
Plugging (\ref{016})-(\ref{017}) into (\ref{015}) yields that
\begin{align}
\nonumber\|m_{n+l+1}(t)-m_{n+1}(t)\|_{B^{s-1}_{p,r}}\leq C_{T}\bigg(\|(S_{n+l+1}-S_{n+1})m_0\|_{B^{s-1}_{p,r}}\\
+\int^t_0 C\mathbf{M}\|m_{n+l}-m_{n}\|_{B^{s-1}_{p,r}}dt'\bigg).
\end{align}

Since
$$\|\sum_{q=n+1}^{n+l}\Delta_{q}m_{0}\|_{B^{s-1}_{p,r}}\leq C2^{-n}\|m_{0}\|_{B^{s-1}_{p,r}} , $$ and that
$(m_n)_{n\in\mathbb{N}}$ is uniformly bounded in $L^{\infty}([0,T];B^{s}_{p,r})$, then it follows that
$$\|m_{n+l+1}(t)-m_{n+1}(t)\|_{B^{s-1}_{p,r}}\leq
C_T(2^{-n}+\int^{t}_{0}\|m_{n+l}-m_{n}\|_{B^{s-1}_{p,r}}d\tau).$$
It is easily checked by induction
$$\|m_{n+l+1}-m_{n+1}\|_{L^{\infty}(0,T; B^{s-1}_{p,r})}\leq \frac{(TC_T)^{n+1}}{(n+1)!}\|m_l-m_{0}\|_{L^{\infty}(0,T; B^{s-1}_{p,r})}
+C_T2^{-n}\sum_{k=1}^{n}2^{k}\frac{(TC_T)^{k}}{k!}.$$
Since $\|m_n\|_{L^{\infty}(0,T; B^{s-1}_{p,r})}$ is bounded independently of $n$, we can find a new constant $C'_T$ such that
$$\|m_{n+l+1}-m_{n+1}\|_{L^{\infty}(0,T; B^{s-1}_{p,r})}\leq C'_T2^{-n}.$$
Consequently, $(m_n)_{n\in\mathbb{N}}$ is a Cauchy sequence in $L^\infty(0,T;B^{s-1}_{p,r})$.
 Moreover it converges to some limit function $m\in L^\infty(0,T;B^{s-1}_{p,r}).$

{\bf Step 4}: We now prove the existence of solution. We prove that $m$ belongs to $E^{s}_{p,r}$ and satisfies (\ref{E2}) in the sense of distribution.
Since $(m_n)_{n\in\mathbb{N}}$ is uniformly bounded in $L^\infty(0,T;B^{s}_{p,r})$, the Fatou property for the Besov spaces guarantees that $m \in L^\infty(0,T;B^{s}_{p,r})$.
\\If $s'\leq s-1,$ then
\begin{align}\label{4}
\|m_n-m\|_{B^{s'}_{p,r}}\leq C\|m_n-m\|_{B^{s-1}_{p,r}}.
\end{align}
If $s-1\leq s'<s,$ by using Lemma \ref{3}, we have
\begin{align}\label{5}
\|m_n-m\|_{B^{s'}_{p,r}}\nonumber &\leq C\|m_n-m\|^\theta_{B^{s-1}_{p,r}}
\|m_n-m\|^{1-\theta}_{B^{s}_{p,r}}\\&
\leq C\|m_n-m\|^\theta_{B^{s-1}_{p,r}}(\|m_n\|_{B^{s}_{p,r}}+\|m\|_{B^{s}_{p,r}})^{1-\theta},
\end{align}
where $\theta=s-s'$.
Combining (\ref{4}) with (\ref{5}) for all $s'< s$, we have that $(m_n)_{n\in\mathbb{N}}$ converges to $m$ in $L^\infty([0,T];B^{s'}_{p,r})$. Taking limit in (\ref{E00}), we conclude that $m$ is indeed a solution of (\ref{E2}). Note that $m\in L^\infty(0,T;B^{s}_{p,r}).$
Then
\begin{align}\label{18}
\nonumber&\|2m^2+(8\partial_xu-4u)m+2(u+\partial_xu)^2\|_{B^{s}_{p,r}}
\\\nonumber \leq&
\|2m^2\|_{B^{s}_{p,r}}+\|(8\partial_xu-4u)m\|_{B^{s}_{p,r}}
+\|2(u+\partial_xu)^2\|_{B^{s}_{p,r}}
\\\nonumber \leq&
C\|m\|_{B^{s}_{p,r}}\|m\|_{L^{\infty}}+C\|m\|_{B^{s}_{p,r}}\|8\partial_xu-4u\|_{L^{\infty}}
\\\nonumber&+C\|8\partial_xu-4u\|_{B^{s}_{p,r}}\|m\|_{L^{\infty}}
+C\|u+\partial_xu\|_{B^{s}_{p,r}}\|u+\partial_xu\|_{L^{\infty}}
\\ \leq&C\|m\|^2_{B^{s}_{p,r}}.
\end{align}
This shows that the right-hand side of (\ref{E2}) also belongs to $L^\infty(0,T;B^{s}_{p,r}).$
Hence, according to Lemma \ref{est2}, $m$ belongs to $C([0,T);B^s_{p,r})~ (resp.,C_w([0,T);B^s_{p,r}))$ if $r<\infty~(resp.,r=\infty)$. Lemma (\ref{Morse}) implies that $(4u-2u_x)m_x$ is bounded in $L^\infty(0,T;B^{s-1}_{p,r}).$
Again using the equation (\ref{E2}) and the high regularity of $u$, we see that $\partial_{t}u$ is in $ C([0,T);B^{s-1}_{p,r})$ if $r$ is finite. Apparently, we know that $m\in E^{s}_{p,r}.$

{\bf Step 5}: Finally, we prove the uniqueness and stability of solutions to (\ref{E2}). Suppose that $M=(1-\partial_x^2)u,~N=(1-\partial_x^2)v \in E^{s}_{p,r}$ are two solutions of (\ref{E2}).
Set $W=M-N.$ Hence, we obtain that

\begin{align}
\left\{
\begin{array}{ll}
&\partial_{t}W-(4u-2\partial_{x}u)\partial_{x}W
 \\&=(u-v)(\partial_{x}G^{1}+G^{2})
 +\partial_{x}(u-v)
 (\partial_{x}G^{3}+G^{4})+W G^{5},\\[1ex]
&W(t,x)|_{t=0}=M(0)-N(0)=W(0),\\[1ex]
\end{array}
\right.
\end{align}
where
\begin{align}
\nonumber&G^{1}=4N+2u+2v,
\\\nonumber&G^{2}=-4N+2u+2v,
\\\nonumber&G^{3}=-2N+2u+2v,
\\\nonumber&G^{4}=8N+2u+2v,
\\\nonumber&G^{5}=2M+2N+8\partial_{x}u-4u.
\end{align}

We define that
$U(t)\triangleq\int^{t}_{0}\|m(t')\|_{B^{s}_{p,r}}dt'$. By (\ref{10}) of Lemma \ref{est1} and using the fact that $m$ is bounded in $L^{\infty}(0,T;B^{s}_{p,r})$, we infer that
\begin{align}\label{15}
\nonumber\|W\|_{B^{s-1}_{p,r}}\leq& Ce^{CU(t)}\bigg{(}\|W(0)\|_{B^{s-1}_{p,r}}+\int^t_0
e^{-CU(t')}(\|(u-v)\partial_{x}G^{1}\|_{B^{s-1}_{p,r}}
+\|(u-v)G^{2}\|_{B^{s-1}_{p,r}}
\\&+\|\partial_{x}(u-v)\partial_{x}G^{3}\|_{B^{s-1}_{p,r}}
+\|\partial_{x}(u-v)G^{4}\|_{B^{s-1}_{p,r}}+\|WG^{5}\|_{B^{s-1}_{p,r}})dt'\bigg{)}.
\end{align}

Taking advantage of Lemma \ref{Morse} with $d=1$, we have
\begin{align}\label{16}
\|(u-v)\partial_{x}G^{1}\|_{B^{s-1}_{p,r}}\nonumber&\leq
\|u-v\|_{B^{s}_{p,r}}\|\partial_{x}(4N+2u+2v)\|_{B^{s-1}_{p,r}}\\\nonumber&\leq
\|W\|_{B^{s-1}_{p,r}}\|4N+2u+2v\|_{B^{s}_{p,r}}\\\nonumber&\leq
C\|W\|_{B^{s-1}_{p,r}}(\|N\|_{B^{s}_{p,r}}
+\|u\|_{B^{s}_{p,r}}+\|v\|_{B^{s}_{p,r}})
\\&\leq C\mathbf{M}\|W\|_{B^{s-1}_{p,r}},
\\
\|(u-v)R_{n,l}^{2}\|_{B^{s-1}_{p,r}}&\nonumber\leq
\|u-v\|_{B^{s-1}_{p,r}}\|-4N+2u+2v\|_{B^{s}_{p,r}}
\\\nonumber&\leq C\|W\|_{B^{s-1}_{p,r}}(\|N\|_{B^{s}_{p,r}}
+\|u\|_{B^{s}_{p,r}}+\|v\|_{B^{s}_{p,r}})
\\&\leq C\mathbf{M}\|W\|_{B^{s-1}_{p,r}},
\\
\|\partial_{x}(u-v)\partial_{x}G^{3}\|_{B^{s-1}_{p,r}}&\nonumber\leq
\|\partial_{x}(u-v)\|_{B^{s}_{p,r}}\|\partial_{x}(-2N+2u+2v)\|_{B^{s-1}_{p,r}}
\\ \nonumber&\leq
 C\|W\|_{B^{s-1}_{p,r}}(\|N\|_{B^{s}_{p,r}}
+\|u\|_{B^{s}_{p,r}}+\|v\|_{B^{s}_{p,r}})
\\ &\leq C\mathbf{M}\|W\|_{B^{s-1}_{p,r}},
\\
\|\partial_{x}(u-v)G^{4}\|_{B^{s-1}_{p,r}}\nonumber&\leq
\|\partial_{x}(u-v)\|_{B^{s}_{p,r}}\|8N+2u+2v\|_{B^{s}_{p,r}}
\\ \nonumber&\leq
 C\|W\|_{B^{s-1}_{p,r}}(\|N\|_{B^{s}_{p,r}}
+\|u\|_{B^{s}_{p,r}}+\|v\|_{B^{s}_{p,r}})
\\ &\leq C\mathbf{M}\|W\|_{B^{s-1}_{p,r}},
\end{align}
\begin{align}\label{17}
\|WG^{5}\|_{B^{s-1}_{p,r}}\nonumber\leq&
\|W\|_{B^{s-1}_{p,r}}\|2M+2N+8\partial_{x}u-4u\|_{B^{s}_{p,r}}
\\\nonumber \leq
& C\|W\|_{B^{s-1}_{p,r}}(\|M\|_{B^{s}_{p,r}}+\|N\|_{B^{s}_{p,r}}
+\|\partial_{x}u\|_{B^{s}_{p,r}}+\|u\|_{B^{s}_{p,r}})
\\ \leq& C \mathbf{M}\|W\|_{B^{s-1}_{p,r}}.
\end{align}
Plugging (\ref{16})-(\ref{17}) into (\ref{15}) yields that
\begin{align}
e^{-CU(t)}\|W\|_{B^{s-1}_{p,r}}\leq C\|W(0)\|_{B^{s-1}_{p,r}}
+\int^t_0 C\mathbf{M}e^{-CU(t')}\|W\|_{B^{s-1}_{p,r}}dt'.
\end{align}

Applying Gronwall's inequality yields
\begin{align}\label{19}
\sup_{t\in[0,T)}\|W(t)\|_{B^{s-1}_{p,r}}\leq e^{\widetilde{C_{T}}}\|W(0)\|_{B^{s-1}_{p,r}}.
\end{align}
In particular, $u_0=v_0$ in (\ref{19}) yields $u(t)=v(t).$
\\{\bf Step 6}: Continuity with respect to the initial data.
If $s'=s-1$, by (\ref{19}), the conclusion is valid. If $s'<s-1$, by using Lemma \ref{3} and (\ref{19}),
we have
\begin{align}
\|M(t)-N(t)\|_{B^{s'}_{p,r}}\nonumber\leq &C\|M(t)-N(t)\|_{B^{s-1}_{p,r}}
\\\nonumber\leq&
e^{\widetilde{C_{T}}}\|M(0)-N(0)\|_{B^{s-1}_{p,r}}
\\\nonumber\leq& Ce^{\widetilde{C_{T}}}\|M(0)-N(0)\|^{\theta_1}_{B^{s'}_{p,r}}
\|M(0)-N(0)\|^{(1-\theta_1)}_{B^{s_1}_{p,r}}
\\\nonumber\leq& Ce^{\widetilde{C_{T}}}\|M(0)-N(0)\|^{\theta_1}_{B^{s'}_{p,r}}
(\|M(0)\|_{B^{s}_{p,r}}+\|N(0)\|_{B^{s}_{p,r}})^{(1-\theta_1)},
\end{align}
where $s-1=\theta_1s'+s_1(1-\theta_1),~s_1<s.$

If $s-1<s'<s$, by using Lemma \ref{3} and (\ref{19}) again, we get
\begin{align}
\|M(t)-N(t)\|_{B^{s'}_{p,r}}\nonumber\leq &C\|M(t)-N(t)\|^{\theta_2}_{B^{s-1}_{p,r}}\|M(t)-N(t)\|^{1-\theta_2}_{B^{s}_{p,r}}
\\\nonumber\leq& C\|M(t)-N(t)\|^{(1-\theta_2)}_{B^{s}_{p,r}}
e^{\widetilde{C_{T}}\theta_2}\|M(0)-N(0)\|^{\theta_2}_{B^{s-1}_{p,r}}
\\\nonumber\leq& Ce^{\widetilde{C_{T}}\theta_2}(\|M(t)\|_{B^{s}_{p,r}}+\|N(t)\|_{B^{s}_{p,r}})^{(1-\theta_2)}
\|M(0)-N(0)\|^{\theta_2}_{B^{s'}_{p,r}},
\end{align}
where $s'=\theta_2(s-1)+(1-\theta_2)s.$

Consequently, we complete the proof of Theorem \ref{Thm1}.

\end{proof}
\section{Blow-up}
After obtaining local well-posedness theory, a natural question is whether the corresponding solution exists globally in time or not. This section is devoted to the blow-up phenomena. We first show the following conservation law to (\ref{E1}).
\begin{lemm}\label{Thm2}
Let $u_0\in H^s,~s>\frac{5}{2}.$ Then the corresponding solution $u$ guaranteed by Theorem (\ref{Thm1}) has constant energy integral
$$\int_{\mathbb{R}}(u^2+u_x^2)dx
=\int_{\mathbb{R}}[u_0^2+(u_0')^2]dx=\|u_0\|^2_{H^1}.$$
\end{lemm}

\begin{proof}
 Arguing by density, it suffices to consider the case where $u\in C^{\infty}_0.$
Applying integration by parts, we obtain
$$\int_{\mathbb{R}}um dx=\int_{\mathbb{R}}u(u- u_{xx}) dx=\int_{\mathbb{R}}u^2dx+\int_{\mathbb{R}} u_x^2dx.$$
By (\ref{E1}), we infer that
\begin{align}
\nonumber&\frac{d}{d_t}\int_{\mathbb{R}}um dx\\\nonumber&=\int_{\mathbb{R}}(\partial_tu m+\partial_tmu)dx
\\\nonumber&=2\int_{\mathbb{R}}\partial_tm udx
=2\int_{\mathbb{R}}u\partial_x(2+\partial_x)[(2-\partial_x)u]^2dx
\\\nonumber&=-2\int_{\mathbb{R}}[(2-\partial_x)u]^2\partial_x(2-\partial_x)udx
\\\nonumber&=-2\int_{\mathbb{R}}[(2-\partial_x)u]^2d(2-\partial_x)u
\\\nonumber&=0.
\end{align}
\end{proof}

It is easy to see that the Degasperis-Procesi equation transforms into the  equation (\ref{E1}) under the transformation
$$u\rightarrow 2(2-\partial_{x})u.$$

The following important estimate can be obtained by Lemma \ref{Thm2} and Lemma 3.2 in \cite{L-Y1}.

\begin{lemm}\label{Thm02}
Let $u_0\in H^s,~s>\frac{5}{2}.$ Let $T$ be the maximal existence time of the solution $u$ guaranteed by Theorem \ref{Thm1}.
Set $w=(2-\partial_{x})u$ and $w_{0}=(2-\partial_{x})u_{0}$.
Then we have
\begin{align}
\|w\|_{L^{\infty}}&\leq 6\|w_{0}(x)\|^2_{L^2}t+\|w_{0}\|_{L^{\infty}},~~~~~~~~\forall t\in [0,T].
\end{align}
\end{lemm}
\begin{rema}\label{l21}
From Lemma \ref{Thm02}, we have
\begin{align}
\|u_{x}\|_{L^\infty}&\leq 54T\|u_{0}\|^2_{H^1}+4\|u_{0}\|_{H^1}+\|u_{0x}\|_{L^\infty}
\nonumber\\&\leq 54T\|u_{0}\|^2_{H^1}+5\|u_{0}\|_{H^\frac{3}{2}}.
\end{align}
\end{rema}

Then, we present a blow-up criterion for (\ref{E1}).
\begin{lemm}\label{Thm3}
Let $u_0(x)\in H^s(\mathbb{R}),~ s> \frac{5}{2},$ and let $T$ be the maximal existence time of the solution $u(x, t)$ to (\ref{E1}) with the initial data $u_0(x)$. Then the corresponding solution blows up in finite time if and only if
$$ \limsup_{t\rightarrow T}\int^{T}_{0}\|u_{xx}\|_{L^\infty}d\tau=\infty.$$
\end{lemm}
\begin{proof}
Arguing by density, it suffices to consider the case where $u\in C^{\infty}_0$.
The equation (\ref{E1}) can be rewritten as
\begin{align}\label{21}
u_t-u_{txx}=-4\partial_{x}^2(uu_x)+2\partial_x(u_xu_{xx})+2\partial_{x}(u_x^2)+16uu_{x}.
\end{align}
Set $\Lambda^2=1-\partial_{x}^2$. Then applying $\Lambda^{s-1}u\Lambda^{s-1}$ to (\ref{21}) yields
\begin{align}\label{22}
&\nonumber\frac{d}{dt}\int_{\mathbb{R}}[(\Lambda^{s-1}u)^2+(\Lambda^{s-1}u_x)^2]dx
\\\nonumber=&-8\int_{\mathbb{R}}\Lambda^{s-1}u\Lambda^{s-1}(\partial_{x}^2(uu_x))dx
+4\int_{\mathbb{R}}\Lambda^{s-1}u\Lambda^{s-1}(\partial_x(u_xu_{xx}))dx
\\\nonumber&+4\int_{\mathbb{R}}\Lambda^{s-1}u\Lambda^{s-1}(\partial_{x}(u_x^2))dx
+32\int_{\mathbb{R}}\Lambda^{s-1}u\Lambda^{s-1}(uu_{x})dx
\\\nonumber=&24\int_{\mathbb{R}}\Lambda^{s-1}u\Lambda^{s-1}(uu_x)dx
-8\int_{\mathbb{R}}\Lambda^{s}u\Lambda^{s}(uu_x)dx
\\&-4\int_{\mathbb{R}}\Lambda^{s-1}u_{x}\Lambda^{s-1}(u_{x}u_{xx}))dx
-4\int_{\mathbb{R}}\Lambda^{s-1}u_{x}\Lambda^{s-1}(u^2_x)dx.
\end{align}
We estimate the terms on the right-hand side of (\ref{22}), respectively. Applying the Cauchy-Schwartz inequality, Lemmas \ref{com.est}, Corollary \ref{est3} and Remark \ref{l21}, we derive
\begin{align}
\nonumber\int_{\mathbb{R}}&\Lambda^{s-1}u\Lambda^{s-1}(uu_x)dx
\\\nonumber=&\int_{\mathbb{R}}\Lambda^{s-1}u[\Lambda^{s-1}(uu_x)-u\Lambda^{s-1}u_x]dx
-\frac{1}{2}\int_{\mathbb{R}}u_x(\Lambda^{s-1}u)^2dx\\\nonumber \leq&
C(\|\Lambda^{s-1}u\|_{L^2}\|u_x\|_{L^{\infty}}+\|\Lambda^{s-2}u_x\|_{L^2}\|u_x\|_{L^{\infty}})
\|\Lambda^{s-1}u\|_{L^2}+\|u_x\|_{L^{\infty}}\|\Lambda^{s-1}u\|^2_{L^2}
\\\nonumber \leq& C \|u\|^2_{H^{s-1}}\|u_x\|_{L^{\infty}}
\\\leq& C_{T}\|u\|^2_{H^{s-1}},
\\\nonumber
\int_{\mathbb{R}}&\Lambda^{s-1}u_{x}\Lambda^{s-1}(u_xu_{xx})dx
\\\nonumber=&\int_{\mathbb{R}}\Lambda^{s-1}u_x[\Lambda^{s-1}(u_xu_{xx})-u_x\Lambda^{s-1}u_{xx}]dx
-\frac{1}{2}\int_{\mathbb{R}}u_{xx}(\Lambda^{s-1}u_x)^2dx
\\\nonumber \leq& C(\|u_{xx}\|_{L^\infty}\|\Lambda^{s-1}u_x\|_{L^2}
+\|u_{xx}\|_{L^\infty}\|\Lambda^{s-2}u_{xx}\|_{L^2})\|\Lambda^{s-1}u_x\|_{L^2}
+\|u_{xx}\|_{L^\infty}\|\Lambda^{s-1}u_x\|^2_{L^2}
\\ \leq& C \|u\|^2_{H^{s}}\|u_{xx}\|_{L^{\infty}},
\\\nonumber
\int_{\mathbb{R}}&\Lambda^{s-1}u_x\Lambda^{s-1}(u_{x}^2)dx\\\nonumber\leq& \|\Lambda^{s-1}u_x\|_{L^2}\|\Lambda^{s-1}u_{x}^2\|_{L^2}
\\\nonumber\leq& C\|u_x\|_{H^{s-1}}\|u_{x}^2\|_{H^{s-1}}
\\\nonumber\leq& C\|u\|^2_{H^{s}}\|u_{x}\|_{L^{\infty}}
\\\leq& C_{T}\|u\|^2_{H^{s}}.
\end{align}
By the same token, we get
$$\int_{\mathbb{R}}\Lambda^{s}u\Lambda^{s}(uu_x)dx\leq C \|u\|^2_{H^{s}}\|u_x\|_{L^{\infty}}\leq C_{T} \|u\|^2_{H^{s}}.$$
Combining (\ref{22}) with the above inequalities and Remark \ref{l21}, we have
\begin{align}\label{23}
\nonumber\frac{d}{dt}& \|u\|^2_{H^{s}}\\\nonumber\leq&
C_{T} \|u\|^2_{H^{s}}\|u_{xx}\|_{L^{\infty}}
\\ \leq& C_{T}  \|u\|^2_{H^{s}}\|u_{xx}\|_{L^{\infty}}.
\end{align}
Applying Gronwall's inequality to (\ref{23}), we have
\begin{align}\label{24}
\|u\|^2_{H^{s}}\leq\|u_0\|^2_{H^{s}}e^{C_{T}\int_{0}^{T}\|u_{xx}\|_{L^{\infty}}d\tau}.
\end{align}
If $\|u_{xx}\|_{L^\infty}$ is bounded, from (\ref{24}), we know that $\|u(t,x)\|_{H^s}$ is bounded. By using the Sobolev embedding theorem $H^s(\mathbb{R})\hookrightarrow L^{\infty}(\mathbb{R}),~ s>\frac{1}{2},$ we have
\begin{align}\label{25}
\|u_{xx}\|_{L^\infty} \leq C\|u\|_{H^s},
\end{align}
with $s > \frac{5}{2}$. If $\|u\|_{H^s}$ is bounded, $s > \frac{5}{2},$ from (\ref{25}), we know that $\|u_{xx}\|_{L^\infty}$ are bounded. Moreover, if the maximal existence
time $T < \infty $ satisfies
$$\int^{T}_{0}\|u_{xx}\|_{L^{\infty}}d\tau<\infty,$$
we obtain from (\ref{24})
$$\lim_{t\rightarrow T}\sup_{0\leq \tau \leq t}\|u\|_{H^s}<\infty,$$
which contradicts the assumption that $T < \infty$ is the maximal existence time.
 Thus we have
$$\int^{T}_{0}\|u_{xx}\|_{L^\infty}d\tau=\infty.$$
\end{proof}

 The following main theorem of this section shows that there are some initial data for which the corresponding solutions to (\ref{E1}) with some certain condition will blow up in finite time.

\begin{theo}\label{thm.b1}
Assume that $u_0\in H^s, ~~s>\frac{5}{2}$  and $~u''_{0}(x_0)<-4(54T\|u_{0}\|^2_{H^1}+6\|u_{0}\|_{H^\frac{3}{2}})$, then the corresponding solution of (\ref{E1}) blows up in finite time.
\end{theo}
\begin{proof}
By a standard density argument, here we may assume $s=3$ to prove the theorem.

Note that $G(x)=\frac{1}{2}e^{-|x|}$ and $G(x)\star f =(1-\partial _x^2)^{-1}f$ for all $f \in L^2$ and $G \star m=u$. Then we can rewrite (\ref{E2}) as follows:
\begin{align}\label{051}
u_t-4uu_x=-u^2_{x}+G\star[\partial_x(2u_x^2+6 u^{2})+u^2_{x}].
\end{align}
Differentiating this relation twice with respect to $x$, we find
\begin{align}\label{26}
&\nonumber\partial_tu_{xx}+(2u_x-4u)u_{xxx}\\\nonumber=&-2u^2_{xx}+8u_{x}u_{xx}-12uu_{x}-u_{x}^2
\\&+G\star\{\partial_x(2u_x^2+6 u^{2})+(u^2_{x})\}.
\end{align}
Note that
$$|G_{x}\star(2u_x^2+6 u^{2})|+|G\star(u_x^2)|\leq 5\|u\|_{H^1}^{2}\leq 5\|u_0\|_{H^1}^{2}.$$

Therefore,
\begin{align}\label{32}
\partial_tu_{xx}+(2u_x-4u)u_{xxx}&\leq -u_{xx}^2+15 u_{x}^{2}-12uu_{x}+5\|u_0\|_{H^1}^{2}
\\\nonumber&\leq -u_{xx}^2+16u_{x}^2+36u^2+5\|u_0\|_{H^1}^{2}
\\\nonumber&\leq -u_{xx}^2+16 \|u_{x}\|_{L^\infty}^{2}+23\|u_0\|_{H^1}^{2}
\\\nonumber&\leq -u_{xx}^2+16(54T\|u_{0}\|^2_{H^1}+5\|u_{0}\|_{H^\frac{3}{2}})^2+23\|u_0\|_{H^1}^{2}
\\\nonumber&\leq -u_{xx}^2+16(54T\|u_{0}\|^2_{H^1}+6\|u_{0}\|_{H^\frac{3}{2}})^2
\\\nonumber&= -u_{xx}^2+\mathbf{C_{T}}^2,
\end{align}
where $\mathbf{C_{T}}\triangleq 4(54T\|u_{0}\|^2_{H^1}+6\|u_{0}\|_{H^\frac{3}{2}}).$

Defining now $w(t):=\inf_{x\in \mathbb{R}}[u_{xx}(t, q(t,x))],~~\frac{dq(t,x)}{dt}=2u_x(t,q(t,x))-4u(t,q(t,x))$, we obtain from the above inequality the relation
\begin{align}\label{27}
\frac{dw}{dt}\leq  -w^2+\mathbf{C_{T}}^{2},~~~~t\in(0,T).
\end{align}
Since $w(0)<-\mathbf{C_{T}}$, it then follows that
$$ w(t)<-\mathbf{C_{T}}, ~~~~~~~~~\forall t\in [0,T).
$$
By solving the inequality (\ref{27}), we get
\begin{align}\label{28}
0\leq \frac{2\mathbf{C_{T}}}
{w+\mathbf{C_{T}}}
\leq 1-\frac{w(0)-\mathbf{C_{T}}}
{w(0)+\mathbf{C_{T}}}e^{-2t\mathbf{C_{T}}}.
\end{align}

By (\ref{28}) and the fact
$\frac{w(0)-\mathbf{C_{T}}}{w(0)+\mathbf{C_{T}}}>1$, there exists
$$0<T\leq -\frac{1}{2\mathbf{C_{T}}}
\ln{\frac{w(0)+\mathbf{C_{T}}}{w(0)-\mathbf{C_{T}}}},
$$
such that
\begin{align}\label{29}
w(t)\leq-\mathbf{C_{T}}+\frac{2\mathbf{C_{T}}}
{1-\frac{w(0)+\mathbf{C_{T}}}{w(0)-\mathbf{C_{T}}}e^{-2\mathbf{C_{T}}t}}\rightarrow -\infty,
\end{align}
as $t\rightarrow T$. This proves that the wave $u(t,x)$ breaks in finite time.
\end{proof}

Finally, we prove the exact blow-up rate for blowing-up solutions to (\ref{E1}). We now introduce the following useful lemma.
\begin{lemm}\cite{Constantin.Escher3}
 Let $T > 0$ and $u \in C^1([0, T );H^2)$. Then for every $t \in [0, T )$,
there exists at least one point $\xi(t) \in \mathbb{R}$ with
$$m(t) \triangleq \inf_{x\in\mathbb{R}}
(u_x(t, x)) = u_{x}(t, \xi(t)).$$
The function $m(t)$ is absolutely continuous on $(0, T)$ with
$$\frac{dm}{dt}= u_{tx}(t, \xi(t)) ~~~~a.e. ~on ~~~(0, T ).$$
\end{lemm}

\begin{theo}\label{thm.b2}
Let $u_0\in H^s,~s>\frac{5}{2}$, and let $T$ be the blow-up time of the corresponding solution $u$ to (\ref{E1}), which is guaranteed by Theorem \ref{thm.b1}. Then
\begin{align}\label{73}
\lim_{t\rightarrow T}(\inf_{x\in\mathbb{R}}[u_{xx}(t,x)](T-t))=-\frac{1}{2}.
\end{align}
\end{theo}
\begin{proof}
As mentioned earlier, we only need to prove the theorem for $s=3.$

Differentiating  (\ref{051}) with respect to $x$ , we find
\begin{align}
\partial_{t}u_x(t,x)+(2u_{x}-4u)u_{xx}&=2u_{x}^2-6u^2+G\star[(2u_x^2+6 u^{2})+\partial_x(u^2_{x})],
\end{align}

which along with (\ref{26}), leads to
\begin{align}\label{052}
\partial_{t}(&u_{xx}-2u_x)(t,x)+(2u_{x}-4u)(u_{xxx}-2u_{xx})\nonumber
\\\nonumber=&-2u_{xx}^2+8u_{x}u_{xx}-5u_{x}^2
+12u^2-12uu_{x}
\\\nonumber&+G\star[(-3u_x^2-12 u^{2})+\partial_x(6 u^{2})],
\\\nonumber=&-2(u_{xx}-2u_{x})^2+3u_{x}^2+12u^2-12uu_{x}
\\&+G\star[(-3u_x^2-12 u^{2})+\partial_x(6 u^{2})].
\end{align}

By (\ref{052}), in view of $|G\star[(-3u_x^2-12 u^{2})+\partial_x(6 u^{2})]|\leq 12\|u_0\|_{H^1}^{2}$, Lemma \ref{Thm02},  $u^2 \leq \frac{1}{2}\|u_{0}\|^2_{H^1}$ and Remark \ref{l21}, we obtain
 \begin{align}
|&\frac{d(u_{xx}-2u_{x})}{dt}+2(u_{xx}-2u_{x})^2|\nonumber
\\\nonumber&\leq 4u_{x}^2+48u^2+12\|u_{0}\|_{H^1}^2
\\\nonumber&\leq 4u_{x}^2+36\|u_{0}\|_{H^1}^2
\\\nonumber&\leq 4(54T\|u_{0}\|^2_{H^1}+5\|u_{0}\|_{H^\frac{3}{2}})^2+36\|u_{0}\|_{H^1}^2
\\\nonumber&\leq 4(54T\|u_{0}\|^2_{H^1}+6\|u_{0}\|_{H^\frac{3}{2}})^2,~~~~\forall t\in(0,T).
\end{align}
Defining now $w(t):=\inf_{x\in \mathbb{R}}[2(u_{xx}-2u_{x})(t, q(t,x))]$, we obtain from the above inequality the relation

\begin{align}\label{30}
|\frac{dw}{dt}+w^2|\leq \mathbf{\widetilde{C_{T}}}^2,~~~~\forall t\in(0,T),
\end{align}
where $\mathbf{\widetilde{C_{T}}}\triangleq 2\sqrt{2}(54T\|u_{0}\|^2_{H^1}+6\|u_{0}\|_{H^\frac{3}{2}}).$

For every $\varepsilon\in(0, \frac{1}{2})$, in view of (\ref{29}), we can find a $t_0\in(0,T)$ such that
$$w(t_0)<-\sqrt{\mathbf{\widetilde{C_{T}}}^2+\frac{\mathbf{\widetilde{C_{T}}}^2}{\varepsilon}}
<-\mathbf{\widetilde{C_{T}}}.
$$
Thanks to (\ref{29}) and (\ref{30}), we have $w(t)<-\mathbf{\widetilde{C_{T}}}.$ This implies that $w(t)$ is decreasing on $[t_0,T),$ hence,
$$w(t)<-\sqrt{\mathbf{\widetilde{C_{T}}}^2+\frac{\mathbf{\widetilde{C_{T}}}^2}{\varepsilon}}
<-\sqrt{\frac{\mathbf{\widetilde{C_{T}}}^2}{\varepsilon}},~~~~\forall t\in[t_0,T).
$$
Noticing that $
-w^2+\mathbf{\widetilde{C_{T}}}^2 \leq \frac{dw(t)}{dt}\leq -w^2+ \mathbf{\widetilde{C_{T}}}^2,~~~a.e.~t\in(t_0,T),
$
we get
\begin{align}\label{71}
-1-\varepsilon\leq \frac{d}{dt}(-\frac{1}{w(t)}) \leq -1+\varepsilon,~~~a.e.~t\in(t_0,T).
\end{align}
Integrating (\ref{71}) with respect to $t\in[t_0, T)$ on $(t,T)$ and applying $\lim_{t\rightarrow T}w(t)=-\infty$ again, we deduce that
\begin{align}\label{72}
(-1-\varepsilon)(T-t)\leq\frac{1}{w(t)}\leq(-1+\varepsilon)(T-t).
\end{align}
Since $\varepsilon\in(0,\frac{1}{2})$ is arbitrary, it then follows from (\ref{72}) that (\ref{73}) holds. Noting that Remark \ref{l21}, we get $\lim_{t\rightarrow T}(\inf_{x\in\mathbb{R}}u_{x}(t,x)(T-t))=0.$

This completes the proof of the theorem.
\end{proof}

\section{Peakon solutions}
In this section, we show that (\ref{E1}) possesses the following peakon solutions
\begin{align}\label{l41}
u(t,x)=\left \{\begin {array}{ll}-\frac{c}{6}e^{ct-x},
\ \ \ \ \ \ \ &x\geq ct,\\-\frac{c}{2}e^{x-ct}+\frac{c}{3}e^{2x-2ct}, \ \ \ \ \ \ & x< ct,
\end {array}\right.
\end{align}

where $c$ is any real positive constant. Of course the above solutions are not classical solutions but weak solutions. We first introduced the definition of weak solutions for (\ref{E1}).
\begin{defi}
Functions $u$ was called weak solutions for (\ref{E1}) if for each $\varphi(t,x)\in C^1([0,T]; C^\infty_0(\mathbb{R}))$, we have
\begin{align}
\nonumber&\int^T_0\int_{\mathbb{R}}\bigg((\varphi_{t}-\varphi_{txx})u -\partial_{x}(2+\partial_{x})\varphi[(2-\partial_{x})u]^2 \bigg)dxdt \\&=\int_{\mathbb{R}}(u(T)-u_{xx}(T))\varphi(T)dx-\int_{\mathbb{R}}(u_{0}(x)-u_{0xx})\varphi(0)dx.
\end{align}
\end{defi}
Directly calculating, we deduce that

\begin{align}\label{050}
(2-\partial_{x})u(t,x)=\left \{\begin {array}{ll}-\frac{c}{2}e^{ct-x},\ \ \ \ \ \ \ &x> ct,\\-\frac{c}{2}e^{x-ct}, \ \ \ \ \ \ & x< ct. \end {array}\right.
\end{align}

Hence, for any $\varphi(t,x)\in C^1([0,T]; C^\infty_0(\mathbb{R}))$, using integration by parts, we obtain
\begin{align}\label{l43}
\nonumber\int_{\mathbb{R}}u\varphi_{t}\mathrm dx=&\int_{-\infty}^{ct}(-\frac{c}{2}e^{x-ct}+\frac{c}{3}e^{2(x-ct)})
\varphi_{t}\mathrm dx+\int_{ct}^{\infty} (\frac{-c}{6}e^{ct-x}\varphi_{t})\mathrm dx
\\\nonumber=&\int_{-\infty}^{ct}\bigg\{\partial_{t}[(-\frac{c}{2}e^{x-ct}+\frac{c}{3}e^{2(x-ct)})\varphi]
-(\frac{c^2}{2}e^{x-ct}-\frac{2c^2}{3}e^{2(x-ct)})\varphi\bigg\} \mathrm dx
\\\nonumber&+\int^{-\infty}_{ct}\bigg\{\partial_{t}(\frac{-c}{6}e^{ct-x}\varphi)
+(\frac{c^2}{6}e^{x-ct})\varphi\bigg\} \mathrm dx
\\\nonumber=&\int_{-\infty}^{ct}\partial_{t}[(-\frac{c}{2}e^{x-ct}+\frac{c}{3}e^{2(x-ct)})\varphi] \mathrm dx
+\int_{-\infty}^{ct}(-\frac{c^2}{2}e^{x-ct}+\frac{2c^2}{3}e^{2(x-ct)})\varphi \mathrm dx
\\&+\int^{-\infty}_{ct} \partial_{t}(\frac{-c}{6}e^{ct-x}\varphi) \mathrm dx
+\int^{-\infty}_{ct}(\frac{c^2}{6}e^{x-ct})\varphi\mathrm dx.
\end{align}

Taking advantage of integration by parts and the fact that (\ref{050}), we get

\begin{align}\label{l44}
\nonumber\int_{\mathbb{R}}\partial_x&(2+\partial_x)\varphi[(2-\partial_x)u]^2 dx
\\\nonumber=&\int^{ct}_{-\infty} \partial_x(2-\partial_x)\varphi\frac{c^2}{4}e^{2(x-ct)} dx
+\int^{\infty}_{ct}\partial_x(2-\partial_x)\varphi\frac{c^2}{4}e^{2(ct-x)} \mathrm dx
\\\nonumber=&\frac{c^2}{4}\int^{ct}_{-\infty} \bigg\{\partial_x [(2-\partial_x)\varphi e^{2(x-ct)}]-2(2-\partial_x)\varphi e^{2(x-ct)} \bigg\}\mathrm dx
\\\nonumber&+\frac{c^2}{4}\int^{\infty}_{ct}\bigg\{\partial_x[(2-\partial_x)\varphi e^{2(ct-x)}]+2(2-\partial_x)\varphi e^{2(ct-x)}\bigg\}\mathrm dx
\\\nonumber=&\frac{c^2}{4}\int^{ct}_{-\infty} (2\partial_x\varphi-4\varphi)\varphi e^{2(x-ct)}\mathrm dx
+\frac{c^2}{4}\int^{\infty}_{ct}(4\varphi-2\varphi_{x}) e^{2(ct-x)} \mathrm dx
\\\nonumber=&-c^2\int^{ct}_{-\infty} \varphi e^{2(x-ct)}\mathrm dx+\frac{c^2}{2}\int^{ct}_{-\infty} [\partial_x(\varphi e^{2(x-ct)})-2\varphi e^{2(x-ct)} ]\mathrm dx
\\\nonumber&+c^2\int^{\infty}_{ct}\varphi e^{2(ct-x)} \mathrm dx-\frac{c^2}{2}\int^{\infty}_{ct}[\partial_{x}(\varphi e^{2(ct-x)})+2\varphi e^{2(ct-x)}]\mathrm dx
\\\nonumber=&\frac{c^2}{4}[4\varphi (ct)-\varphi_{x}(ct)]-2c^2\int^{Ct}_{-\infty} \varphi e^{2(x-ct)}\mathrm dx
+\frac{c^2}{4}\varphi_{x} (ct)\mathrm dx
\\=&c^2\varphi(ct)-2c^2\int^{ct}_{-\infty} \varphi e^{2(x-ct)}\mathrm dx.
\end{align}

By (\ref{l41}), we obtain

\begin{align}\label{l42}
\nonumber-\int_{\mathbb{R}}u\varphi_{txx}=&-\int_{-\infty}^{ct}(-\frac{c}{2}e^{x-ct}+\frac{c}{3}e^{2(x-ct)})
\varphi_{txx}\mathrm dx-\int_{ct}^{\infty} (\frac{-c}{6}e^{ct-x}\varphi_{txx})\mathrm dx
\\\nonumber=&\int_{-\infty}^{ct}\frac{c}{2}e^{x-ct}\varphi_{txx}\mathrm dx-\int_{-\infty}^{ct}\frac{c}{3}e^{2(x-ct)}
\varphi_{txx}\mathrm dx+\int_{ct}^{\infty} (\frac{c}{6}e^{ct-x}\varphi_{txx})\mathrm dx
\\=&\sum_{i=1}^{3}I_{i},
\end{align}
where
\begin{align}
&I_{1}=\int_{-\infty}^{ct}\frac{c}{2}e^{x-ct}\varphi_{txx}\mathrm dx,\nonumber\\ \nonumber
&I_{2}=-\int_{-\infty}^{ct}\frac{c}{3}e^{2(x-ct)} \varphi_{txx}\mathrm dx,\\ \nonumber
&I_{3}=\int_{ct}^{\infty} (\frac{c}{6}e^{ct-x}\varphi_{txx})\mathrm dx.
\end{align}

An application of integration by parts, yields
\begin{align}
\nonumber I_{1}=&\frac{c}{2}\int_{-\infty}^{ct}e^{x-ct}\varphi_{txx}\mathrm dx
\\\nonumber=&\frac{c}{2}\int_{-\infty}^{ct}[\partial_{x}(e^{x-ct}\varphi_{tx})-e^{x-ct}\varphi_{tx}]\mathrm dx
\\\nonumber=&\frac{c}{2}\varphi_{tx}(ct)-
\frac{c}{2}\int_{-\infty}^{ct}[\partial_{x}(e^{x-ct}\varphi_{t})-e^{x-ct}\varphi_{t}]\mathrm dx
\\\nonumber=&\frac{c}{2}\varphi_{tx}(ct)-\frac{c}{2}\varphi_{t}(ct)+
\frac{c}{2}\int_{-\infty}^{ct}e^{x-ct}\varphi_{t}\mathrm dx
\\=&\frac{c}{2}\varphi_{tx}(ct)-\frac{c}{2}\varphi_{t}(ct)
+\frac{c}{2}\int_{-\infty}^{ct}\partial_{t}(e^{x-ct}\varphi)\mathrm dx
+\frac{c^2}{2}\int_{-\infty}^{ct}e^{x-ct}\varphi\mathrm dx,
\\\nonumber I_{2}=&-\frac{c}{3}\int_{-\infty}^{ct}e^{2(x-ct)} \varphi_{txx}\mathrm dx
\\\nonumber=&-\frac{c}{3}\int_{-\infty}^{ct}[\partial_{x}(e^{2(x-ct)} \varphi_{tx})-2e^{2(x-ct)} \varphi_{tx}]\mathrm dx
\\\nonumber=&-\frac{c}{3}\varphi_{tx}(ct)+\frac{2c}{3}\int_{-\infty}^{ct}[\partial_{x}(e^{2(x-ct)} \varphi_{t})-2e^{2(x-ct)} \varphi_{t}]\mathrm dx
\\\nonumber=&-\frac{c}{3}\varphi_{tx}(ct)+\frac{2c}{3}\varphi_{t}(ct)-\frac{4c}{3}\int_{-\infty}^{ct}e^{2(x-ct)} \varphi_{t}\mathrm dx
\\=&-\frac{c}{3}\varphi_{tx}(ct)+\frac{2c}{3}\varphi_{t}(ct)
-\frac{4c}{3}\int_{-\infty}^{ct}\partial_{t}(e^{2(x-ct)} \varphi)\mathrm dx
-\frac{8c^2}{3}\int_{-\infty}^{ct}(e^{2(x-ct)} \varphi)\mathrm dx,
\\\nonumber I_{3}=&\frac{c}{6}\int_{ct}^{\infty} (e^{ct-x}\varphi_{txx})\mathrm dx
\\\nonumber=&\frac{c}{6}\int_{ct}^{\infty} [\partial_{x}(e^{ct-x}\varphi_{tx})+e^{ct-x}\varphi_{tx}]\mathrm dx
\\\nonumber=&-\frac{c}{6}\varphi_{tx}(ct)
+\frac{c}{6}\int_{ct}^{\infty} \partial_{x}(e^{ct-x}\varphi_{t})+e^{ct-x}\varphi_{t}]\mathrm dx
\\\nonumber=&-\frac{c}{6}\varphi_{tx}(ct)-\frac{c}{6}\varphi_{t}(ct)+\frac{c}{6}\int_{ct}^{\infty} e^{ct-x}\varphi_{t}\mathrm dx
\\\nonumber=&-\frac{c}{6}\varphi_{tx}(ct)-\frac{c}{6}\varphi_{t}(ct)+\frac{c}{6}\int_{ct}^{\infty} \partial_{t}(e^{ct-x}\varphi)\mathrm dx-\frac{c^2}{6}\int_{ct}^{\infty}e^{ct-x}\varphi\mathrm dx.
\end{align}

Combining the above three equalities with (\ref{l42}), we have
\begin{align}
-\int_{\mathbb{R}}u\varphi_{txx}=&\frac{c}{2}\int_{-\infty}^{ct}\partial_{t}(e^{x-ct}\varphi)\mathrm dx
+\frac{c^2}{2}\int_{-\infty}^{ct}e^{x-ct}\varphi\mathrm dx\\\nonumber&-\frac{4c}{3}\int_{-\infty}^{ct}\partial_{t}(e^{2(x-ct)} \varphi)\mathrm dx
-\frac{8c^2}{3}\int_{-\infty}^{ct}(e^{2(x-ct)} \varphi)\mathrm dx\\\nonumber&+\frac{c}{6}\int_{ct}^{\infty} \partial_{t}(e^{ct-x}\varphi)\mathrm dx-\frac{c^2}{6}\int_{ct}^{\infty}e^{ct-x}\varphi\mathrm dx,
\end{align}
which along with (\ref{l43}) and (\ref{l44}), implies
\begin{align}\label{l45}
\nonumber\int_{\mathbb{R}}&\bigg((\varphi_{t}-\varphi_{txx})u
 -\partial_{x}(2+\partial_{x})\varphi[(2-\partial_{x})u]^2 \bigg)dx
\\\nonumber&=-c\int_{-\infty}^{ct}\partial_{t}(e^{2(x-ct)}\varphi)\mathrm dx-c^2\varphi(ct)
\\\nonumber&=-c\partial_{t}\int_{-\infty}^{ct}(e^{2(x-ct)}\varphi)\mathrm dx
\\&\nonumber=\partial_{t}\int_{-\infty}^{ct}(u-u_{xx})\varphi\mathrm dx
\\&=\partial_{t}\int_{\mathbb{R}}(u-u_{xx})\varphi\mathrm dx.
\end{align}

Integrating (\ref{l45}) over $[0,T]$ with respect to $t$ and using integration by parts, we get

\begin{align}
\nonumber&\int^T_0\int_{\mathbb{R}}\bigg((u-u_{xx})\varphi_t -\partial_{x}(2-\partial_{x})\varphi[(2-\partial_{x})u]^2 \bigg)dxdt \\&=\int_{\mathbb{R}}(u(T)-u_{xx}(T))\varphi(T)dx-\int_{\mathbb{R}}(u_{0}(x)-u_{0xx})\varphi(0)dx.
\end{align}
Thus we verify that (\ref{l41}) are weak solutions of (\ref{E1}).\\

\noindent\textbf{Acknowledgements.} This work was
partially supported by NNSFC (No.11271382), RFDP (No.
20120171110014), the Macao Science and Technology Development Fund (No. 098/2013/A3) and the key project of Sun Yat-sen University. The authors thank the referees for their valuable comments
and suggestions.

\phantomsection
\addcontentsline{toc}{section}{\refname}
%添加参考文献到书签，宏包 hyperref

\end{document}